\documentclass[11pt]{amsart}

\usepackage{amsmath}
\usepackage{amssymb}
\usepackage{amscd}

\usepackage[backend=bibtex]{biblatex}
\bibliography{../biblib} 

\usepackage{tikz-cd}
\usepackage{enumerate}

\usepackage{xcolor}

\newcommand{\nc}{\newcommand}

\nc{\kk}{{\mathsf{k}}}

\nc{\CC}{{\mathbb{C}}}
\nc{\PP}{{\mathbb{P}}}
\nc{\QQ}{{\mathbb{Q}}}
\nc{\RR}{{\mathbb{R}}}
\nc{\ZZ}{{\mathbb{Z}}}

\nc{\CA}{{\mathcal{A}}}
\nc{\CB}{{\mathcal{B}}}
\nc{\C}{{\mathcal{C}}}
\nc{\D}{{\mathcal{D}}}
\nc{\CE}{{\mathcal{E}}}
\nc{\CF}{{\mathcal{F}}}
\nc{\CG}{{\mathcal{G}}}
\nc{\CH}{{\mathcal{H}}}
\nc{\CL}{{\mathcal{L}}}
\nc{\CM}{{\mathcal{M}}}
\nc{\CN}{{\mathcal{N}}}
\nc{\CO}{{\mathcal{O}}}
\nc{\CQ}{{\mathcal{Q}}}
\nc{\CR}{{\mathcal{R}}}
\nc{\CS}{{\mathcal{S}}}
\nc{\CT}{{\mathcal{T}}}
\nc{\CU}{{\mathcal{U}}}
\nc{\CV}{{\mathcal{V}}}
\nc{\CW}{{\mathcal{W}}}
\nc{\CX}{{\mathcal{X}}}
\nc{\CY}{{\mathcal{Y}}}

\nc{\eps}{\varepsilon}

\nc{\lotimes}{\mathbin{\mathop{\otimes}\limits^{\mathbb{L}}}}
\nc{\CEnd}{\mathop{\mathcal{E}\mathit{nd}}\nolimits}
\nc{\CExt}{\mathop{\mathcal{E}\mathit{xt}}\nolimits}
\nc{\CHom}{\mathop{\mathcal{H}\mathit{om}}\nolimits}
\nc{\RGamma}{\mathop{{\mathsf{R}}\Gamma}\nolimits}
\nc{\RHom}{\mathop{\mathsf{RHom}}\nolimits}
\nc{\RCHom}{\mathop{\mathsf{R}\mathcal{H}\mathit{om}}\nolimits}
\nc{\RG}{\mathop{\mathsf{R\Gamma}}\nolimits}
\nc{\Hom}{\mathop{\mathsf{Hom}}\nolimits}
\nc{\Ext}{\mathop{\mathsf{Ext}}\nolimits}
\nc{\End}{\mathop{\mathsf{End}}\nolimits}
\nc{\Tor}{\mathop{\mathsf{Tor}}\nolimits}
\nc{\Tordim}{\mathop{\mathsf{Tor}\text{\rm-}\mathsf{dim}}\nolimits}
\nc{\Hilb}{\mathop{\mathsf{Hilb}}\nolimits}
\nc{\Spec}{\mathop{\mathsf{Spec}}\nolimits}
\nc{\Pic}{\mathop{\mathsf{Pic}}\nolimits}

\nc{\Tr}{\mathop{\mathsf{Tr}}\nolimits}
\nc{\Cone}{\mathop{\mathsf{Cone}}\nolimits}
\nc{\Ker}{\mathop{\mathsf{Ker}}\nolimits}
\nc{\Coker}{\mathop{\mathsf{Coker}}\nolimits}
\nc{\codim}{\mathop{\mathsf{codim}}\nolimits}
\nc{\sing}{{\mathsf{sing}}}
\nc{\supp}{\mathop{\mathsf{supp}}}
\nc{\vol}{\mathop{\mathsf{vol}}\nolimits}
\nc{\ch}{\mathop{\mathsf{ch}}\nolimits}
\nc{\perf}{{\mathsf{perf}}}
\nc{\rank}{\mathop{\mathsf{rank}}}

\nc{\Pf}{{\mathsf{Pf}}}
\nc{\Gr}{{\mathsf{Gr}}}
\nc{\OGr}{{\mathsf{OGr}}}
\nc{\LGr}{{\mathsf{LGr}}}
\nc{\IGr}{{\mathsf{IGr}}}
\nc{\OF}{{\mathsf{OF}}}
\nc{\Fl}{{\mathsf{Fl}}}
\nc{\Bl}{{\mathsf{Bl}}}
\nc{\GL}{{\mathsf{GL}}}
\nc{\SO}{{\mathsf{SO}}}
\nc{\PGL}{{\mathsf{PGL}}}
\nc{\SL}{{\mathsf{SL}}}
\nc{\SP}{{\mathsf{Sp}}}
\nc{\Spin}{{\mathsf{Spin}}}
\nc{\Tot}{{\mathsf{Tot}}}
\nc{\ev}{{\mathsf{ev}}}
\nc{\od}{{\mathsf{odd}}}
\nc{\coev}{{\mathsf{coev}}}
\nc{\id}{{\mathsf{id}}}
\nc{\opp}{{\mathsf{opp}}}
\nc{\tdim}{\mathop{\Tor\dim}}
\nc{\ad}{{\mathop{\mathsf ad}}}
\nc{\sg}{{\mathop{\mathsf sg}}}
\nc{\hf}{{\mathop{\mathsf hf}}}
\nc{\gr}{{\mathop{\mathsf gr}}}
\nc{\qgr}{{\mathop{\mathsf qgr}}}
\nc{\Coh}{{\mathop{\mathsf{Coh}}}}
\nc{\Rep}{{\mathop{\mathsf{Rep}}}}
\nc{\fsl}{{\mathfrak{sl}}}
\nc{\fso}{{\mathfrak{so}}}
\nc{\fgl}{{\mathfrak{gl}}}

\theoremstyle{plain}

\newtheorem*{theorem*}{Theorem}
\newtheorem{theorem}{Theorem}[section]

\newtheorem{lemma}[theorem]{Lemma}
\newtheorem{proposition}[theorem]{Proposition}
\newtheorem{corollary}[theorem]{Corollary}

\theoremstyle{definition}

\newtheorem{definition}[theorem]{Definition}

\theoremstyle{remark}

\newtheorem{remark}[theorem]{Remark}
\newtheorem{example}[theorem]{Example}

\nc{\Irr}{\mathop{\mathsf{Irr}}\nolimits}
\nc{\sG}{\mathsf{G}}
\nc{\sL}{\mathsf{L}}
\nc{\sP}{\mathsf{P}}
\nc{\sU}{\mathsf{U}}
\nc{\sT}{\mathsf{T}}
\nc{\sB}{\mathsf{B}}
\nc{\thalf}{\tfrac{1}{2}}

\title{Irreducible Ulrich bundles on isotropic Grassmannians}
\author{Anton Fonarev}
\address{\sloppy
    \parbox{0.9\textwidth}{
        Algebraic Geometry Section, Steklov Mathematical Institute,
        8 Gubkin str., Moscow 119991 Russia
        \hfill
}\bigskip}
\email{avfonarev@mi.ras.ru}
\date{}
\thanks{This work is supported by the RSF under a grant 14-50-00005.}

\begin{document}

\begin{abstract}
    We classify irreducible equivariant Ulrich vector bundles on isotropic Grassmannians.
\end{abstract}

\maketitle

\section{Introduction}
\label{sec:intro}

Ulrich bundles were introduced in~\cite{ulrichEis} in order to study Chow
forms. The notion of an Ulrich module appeared much earlier in commutative
algebra. Ulrich himself gave a certain sharp upper bound on the minimal number
of generators for a maximal Cohen--Macaulay module module over a Cohen--Macaulay
homogeneous ring. Ulrich modules are precisely those for which the upper
bound is attained (see~\cite{ulrich}). There are several equivalent definitions of Ulrich bundles.
Though we are going to use a less enlightning cohomological one, it might
be motivating to formulate the most geometric and, to our taste, the easiest one.
Given a $d$-dimensional projective variety $X\subset \PP^N$, 
an \emph{Ulrich bundle} on $X$ is a vector bundle $E$ such that $\pi_* E$
is a trivial bundle for a general linear projection $\pi:X\to \PP^d$.
Further details may be found in~\cite{ulrichEis}.

It was asked in~\cite{ulrichEis} whether every projective variety admits
an Ulrich bundle. So far the answer is known in few cases which include
hypersurfaces and complete intersections~\cite{ulrichCI}, del~Pezzo
surfaces~\cite{ulrichEis}, and abelian surfaces~\cite{ulrichAbS}.

While it is not clear whether the answer
to the general question is positive, in the case of rational homogeneous
varieties one has a large test class of vector bundles, namely equivariant ones.
In~\cite{ulrichGr} the authors fully classified irreducible equivariant
Ulrich vector bundles on Grassmannians over an algebraically closed field
of characteristic zero. These results were further improved in~\cite{ulrichFl},
where most of partial flag varieties were treated. In the current
paper we move in an orthogonal (or symplectic) direction and classify
irreducible equivariant Ulrich bundles on isotropic Grassmannians, that is
varieties of the form $\sG/\sP$, where $\sG$ is a classical group of type
$B_n$, $C_n$, or $D_n$, and $\sP$ is a maximal parabolic subgroup.

The following meta theorem is the main result of the present paper.
\begin{theorem*}
    The only isotropic Grassmannians admitting an equivariant irreducible
    Ulrich vector bundle are the symplectic Grassmannians of planes
    $\IGr(2,2n)$ for $n\geq 2$, odd and even-dimensional quadrics $Q^d$, 
    orthogonal Grassmannians of planes $\OGr(2,m)$ for $m\geq 4$,
    even orthogonal Grassmannians of $3$-spaces $\OGr(3, 4q+6)$ for $q\geq 0$,
    and $\OGr(4,8)$. In each case the corresponding bundles are classified.
\end{theorem*}

The paper is organized as follows. In Section~\ref{sec:prel} we collect
some preliminary definitions and provide a simple criterion for an equivariant
irreducible vector bundle to be Ulrich. In Section~\ref{sec:igr} we treat all
isotropic Grassmannians in type $C_n$ but maximal (Lagrangian) ones.
Similarly, in sections~\ref{sec:ogr} and~\ref{sec:dn} we deal with all but
maximal Grassmannians in types $B_n$ and $D_n$ respectively.
Finally, Section~\ref{sec:max} is devoted to maximal Grassmannians,
both symplectic and orthogonal.

The author is gratful to the anonymous reviewer for careful reading of
the original manuscript.

\section{Preliminaries}
\label{sec:prel}

Let $\sG$ be a simple algebraic group of type $B_n$, $C_n$, or $D_n$
over an algebraically closed field of characteristic zero.
We fix a maximal torus $\sT\subset\sG$ along with a Borel subgroup $\sB\supset\sT$.
Let $P_\sG$ and $W$ denote the weight lattice and the Weil group
of $\sG$ respectively. In the following all the roots and weights will be expressed
in terms of the standard orthonormal basis
$\langle e_1,e_2,\ldots, e_n \rangle$ in a fixed vector space $\RR^n$.
Denote by $\Phi\supset\Phi^+\supset\Delta$ the root system of $\sG$,
the set of positive roots, and the set of simple roots respectively.
We use the standard numbering of the simple roots $\Delta=\{\eps_1,\ldots,\eps_n\}$.

Our goal is to study vector bundles on
varieties of the form $X=\sG/\sP$, where $\sP=\sP_k$ is the
maximal parabolic subgroup associated with the set $\Delta\setminus \{\eps_k\}$.
We call such varieties \emph{isotropic Grassmannians}.
In particular, maximal isotropic Grassmannians correspond to $\Delta\setminus \{\eps_n\}$.

It is well known that the category of $\sG$-equivariant vector bundles on
$X$ is equivalent to the category of representations of $\sP$:
\[
    \Coh^\sG(\sG/\sP)\simeq \Rep-\sP,
\]
see~\cite{bonKapHomog}. Moreover, this is an equivalence of tensor abelian categories.

The group $\sP$ is not reductive, however, we are interested only in irreducible
equivariant vector bundles on $X$. Those correspond to irreducible representations
of $\sP$, which in turn are all induced from the reductive Levi quotient group $\sL=\sP/\sU$
(here $\sU$ denotes the unipotent radical of $\sP$).

The weight lattice $P_\sL$ is canonically isomorphic to $P_\sG$.
Let $P^+_\sL\subset P_\sL$ denote the cone of $\sL$-dominant weights.
It is generated by $\omega_1,\ldots,\omega_n$ and $-\omega_k$.
For a given $\lambda\in P^+_\sL$ we denote by
$\CU^\lambda$ the corresponding irreducible equivariant vector bundle on $X$.
In particular, $\CU^{\omega_k}$ is the ample generator of $\Pic X\simeq \ZZ$.

The Borel--Bott--Weil theorem is the tool we use to compute cohomology of irreducible
vector bundles on Grassmannians (see, e.g.,~\cite{weyman}).
Let $\rho$ denote the half-sum of positive roots of $\sG$, which also equals
the sum of fundamental weights.
Recall that a weight is called \emph{singular} if it lies on a wall
of a Weyl chamber, i.e. is fixed by a nontrivial element of $W_\sG$. Equivalently,
a weight is singular if and only if it is orthogonal to some (positive) root of $\sG$.

\begin{theorem}[Borel--Bott--Weil]
    \label{thm:bbw}
    Let $\lambda\in P^+_\sL$. If $\lambda+\rho$ is singular, then
    \[
        H^\bullet(X,\,\CU^\lambda) = 0.
    \]
    Otherwise, $\lambda+\rho$ lies in the interior of some Weyl chamber,
    and there exists a unique element $w\in W$ such that
    $w(\lambda+\rho)\in P^+_\sG$. Let $l$ denote the length of $w$.
    Then the only nontrivial cohomology 
    \[
        H^l(X,\,\CU^\lambda) = V^{w(\lambda+\rho)-\rho}
    \]
    is an irreducible representation of $\sG$ of highest weight
    $w(\lambda+\rho)-\rho$.
\end{theorem}

Now that we have all necessary information about irreducible equivariant
vector bundles on isotropic Grassmannians, we can turn to Ulrich bundles.
The latter were introduced in~\cite{ulrichEis} where the authors study
Chow forms. We are going to use the original cohomological characterization.

\begin{definition}
    Let $X\subset \PP^N$ be a projective variety of dimension $d$.
    A vector bundle $\CE$ on $X$ is called \emph{Ulrich}
    if $H^i(X, \CE(t))=0$ for $0 < i < d$ and all $t$, while
    $H^0(X, \CE(t))=0$ for $t < 0$, and $H^d(X, \CE(t))=0$ for $t\geq -d$.
\end{definition}

The goal of the present paper is to classify all irreducible equivariant
Ulrich bundles on $X=\sG/\sP_{\eps_k}$ under the natural embedding
given by $\CU^{\omega_k}$. Let $d=\dim X$.
We are now going to reduce the problem to a~representation-theoretic one.
For a given $\sL$-dominant weight $\lambda\in P^+_\sL$ define
\[
    \Irr(\lambda) = \left\{t\in\ZZ \mid \lambda+\rho-t\omega_k\ \text{is singular}\right\}.
\]

\begin{lemma}
    \label{lm:maxt}
    The set $\Irr(\lambda)$ contains at most $d$ elements.
\end{lemma}
\begin{proof}
    The weight $\lambda=\sum x_i\omega_i$ is $L$-dominant if and only
    if $x_i\geq 0$ for all $i\neq k$. Thus, for $\lambda\in P^+_\sL$
    the coefficients $y_i$ in the decomposition
    $\lambda+\rho=\sum y_i\omega_i$ are positive for all $i\neq k$.
    We are interested in all $t\in\ZZ$ for which
    $\lambda+\rho-t\omega_k$ is orthogonal to a positive root of $\sG$.
    Put 
    \[
        \Phi'=\{\alpha\in\Phi^+\mid \alpha=\sum_{i\neq k} p_i\eps_i\}.
    \]
    For any $\alpha\in\Phi'$ one has $(\alpha, \omega_k)=0$, thus
    $(\alpha, \lambda+\rho-t\omega_k)=\sum_{i\neq k}y_i(\alpha, \omega_i)>0$ for any $t$.
    Meanwhile, for any $\alpha\in\Phi^+\setminus\Phi'$ one has
    $(\alpha, \omega_k)>0$. Thus, there exists a unique \emph{rational} $t$
    such that $(\alpha, \lambda+\rho-t\omega_k)=0$. We conclude that $\Irr(\lambda)$
    contains at most $|\Phi^+\setminus \Phi'|$ elements. It is well known that
    $d=|\Phi^+\setminus \Phi'|$, which finishes the proof.
\end{proof}

We can now formulate a simple criterion for an irreducible equivariant
vector bundle on an isotropic Grassmannian to be Ulrich.
\begin{lemma}
    \label{lm:ulr}
    An irreducible equivariant vector bundle $\CU^\lambda$ on $X$ is Ulrich
    if and only if
    \[
        \Irr(\lambda) = \{1,2,\ldots,d\}.
    \]
\end{lemma}

\begin{proof}
    Given an irreducible equivariant vector bundle $\CU^\lambda$, one has
    $\CU^\lambda(t)\simeq\CU^{\lambda+t\omega_k}$. It follows from the
    Borel--Bott--Weil theorem
    that $\Irr(\lambda) = \{1,2,\ldots,d\}$ if and only if
    $H^\bullet(X,\CU^\lambda(t))=0$ for $-d\leq t < 0$. Thus, the only if part.
    
    In Lemma~\ref{lm:maxt} we have shown that there can be at most $d$ rational
    intersection points of the line $\lambda+\rho+t\omega_k$ with the walls of the
    Weyl chambers. We conclude that if $\Irr(\lambda)=\{1,2,\ldots,d\}$,
    as the maximum number of intersections is reached,
    the ray $\{\lambda+\rho+t\omega_k\}_{t\geq 0}$
    lies in the interior of a single Weyl chamber.
    
    It follows from the Borel--Bott--Weil theorem that an equivariant vector
    bundle can have at most one nontrivial cohomology group, and $\CU^\mu$
    and $\CU^\nu$ have nontrivial cohomology in the same degree if
    $\mu+\rho$ and $\nu+\rho$ lie in the same Weyl chamber.
    As $H^0(X,\CU^\lambda(t))\ne 0$
    for $t\gg 0$, we deduce that $H^0(X, \CU^\lambda(t))\neq 0$ 
    and $H^i(X,\CU^\lambda(t))=0$ for all
    $t\geq 0$ and $0<i\leq d$.
    A similar argument shows that $H^i(X,\CU^\lambda(t))=0$ for all
    $t<-d$ and $0\leq i < d$.
\end{proof}

\section{Symplectic Grassmannians}
\label{sec:igr}

Let $2\leq k\leq n$ be integer numbers and let $V$ be a $2n$-dimensional vector space
equipped with a non-degenerate symplectic
form.\footnote{The case $k=1$ is trivial as $X\simeq\PP(V)$.}
Let $\sG=\SP(V)$ be the corresponding symplectic group.
We identify $P_\sG$ with the set of integral vectors $\ZZ^n\subset\RR^n$.
The root system~$\Phi$ consists of the vectors $\pm 2e_i$ and $\pm (e_i\pm e_j)$ for $i\neq j$,
the simple roots are $\eps_i=e_i-e_{i+1}$ for $i=1,\ldots,n-1$ and $\eps_n=2e_n$,
and the fundamental weights are $\omega_i=e_1+\ldots+e_i$ for $i=1,\ldots,n$.
In particular,
\[
    \rho=(n,n-1,\ldots,1).
\]

Let $X=\sG/\sP$, where $\sP=\sP_k$ is the maximal parabolic subgroup associated with
$\Delta\setminus\{\eps_k\}$.
The variety $X$ is nothing but the Grassmannian $\IGr(k, V)$ of $k$-dimensional
subspaces in $V$ isotropic with respect to the form.
The dimension of $X$ equals $d=k(2n-k)-k(k-1)/2$.

The $\sL$-dominant cone $P^+_\sL$ is identified with $\lambda\in\ZZ^n$ for which
\begin{equation}
    \label{eq:igr:l_dom}
    \lambda_1\geq\lambda_2\geq\ldots\geq\lambda_k,\qquad
    \lambda_{k+1}\geq\lambda_{k+2}\geq\ldots\geq\lambda_n\geq 0.
\end{equation}
For example,
if one denotes by $\CU$ the tautological rank $k$ subbundle, and
$\lambda = (\lambda_1, \ldots, \lambda_k,0,\ldots, 0)\in P^+_\sL$ is such that
$\lambda_k\geq 0$, then $\CU^\lambda\simeq \Sigma^\lambda\CU^*$,
where $\Sigma^\lambda$ is the Schur functor corresponding to $\lambda$.
In particular, $\CO_X(1)=\CU^{\omega_k}\simeq \Lambda^k\CU^*$.

For a given $\lambda\in P^+_\sL$ define auxiliary sequences
$\alpha\in\ZZ^k$ and $\beta\in\ZZ^{n-k}$ by
\begin{equation}
    \label{eq:igr_ab}
    (\alpha_1,\ldots,\alpha_k,\beta_1,\ldots,\beta_{n-k})=\lambda+\rho.
\end{equation}
Condition~\eqref{eq:igr:l_dom} now reads
\begin{equation}
    \label{eq:ab_dom}
    \alpha_1 > \alpha_2 > \ldots > \alpha_k,\qquad \beta_1 > \beta_2 > \ldots > \beta_{n-k} > 0.
\end{equation}

\begin{proposition}
    \label{prop:igr:irr}
    \[
        \Irr(\lambda)=\{\alpha_i\pm \beta_j\}_{1\leq i\leq k,\ 1\leq j\leq n-k}
        \ \bigcup
        \ \ZZ \cap \left\{\frac{\alpha_i+\alpha_j}{2}\right\}_{1\leq i\leq j\leq k}
    \]
\end{proposition}
\begin{proof}
    The weight
    \[
        \lambda+\rho-t\omega_k =
        (\alpha_1-t, \alpha_2-t, \ldots, \alpha_l-t, \beta_1, \beta_2,\ldots, \beta_{n-k})
    \]
    is singular if and only if it is orthogonal to one of the roots of $\sG$.
    The roots are $\pm 2e_i$ and $\pm (e_i\pm e_j)$ for $i\neq j$. In particular,
    a weight is singular if and only if (i) one of the terms equals zero, or (ii) two different
    terms coincide, or (iii) a couple of terms sum up to zero.
    It follows from~\eqref{eq:ab_dom} that
    the three options are equivalent to (i) $t=\alpha_i$, (ii) $t=\alpha_i-\beta_j$, and (iii)
    $t=(\alpha_i+\alpha_j)/2$ or $t=\alpha_i+\beta_j$.
\end{proof}

\begin{remark}
    \label{rem:igr_dist}
    From Lemma~\ref{lm:maxt} we deduce that if $\CU^\lambda$ is Ulrich, then all the numbers
    $\alpha_i\pm \beta_j$ and $(\alpha_i+\alpha_j)/2$ are distinct and integer.
\end{remark}

\subsection{General isotropic Grassmannians}
\label{ssec:igr}
The goal of this section is to prove the following statement.
\begin{proposition}
    \label{prop:igr}
    There are no irreducible equivariant Ulrich bundles on $\LGr(k, V)$ for $2<k<n$.
\end{proposition}

\begin{lemma}
    \label{lm:igr_dim}
    If there exists an irreducible equivariant Ulrich bundle on $X$, then $\dim X$ is odd.
\end{lemma}
\begin{proof}
    Assume that $\Irr(\lambda)=\{1,2,\ldots,d\}$. From~\eqref{eq:ab_dom} and
    Proposition~\ref{prop:igr:irr} we get
    \begin{equation*}
        \min \Irr(\lambda) = \alpha_k-\beta_1 = 1
        \qquad\text{and}\qquad
        \max \Irr(\lambda) = \alpha_1+\beta_1 = d.
    \end{equation*}
    Now, according to Remark~\ref{rem:igr_dist},
    \[
        \ZZ\supset\Irr(\lambda)\ni \frac{\alpha_1+\alpha_k}{2} =
        \frac{(\alpha_1+\beta_1)+(\alpha_k-\beta_1)}{2} = \frac{d+1}{2}.
    \]
    Thus, $d$ is necessarily odd.
\end{proof}

\begin{proof}[Proof of Proposition~\ref{prop:igr}]
    Let us introduce some temporary notation: for integers $a$ and $b$ we will write
    $a\prec b$ if $b=a+1$, and similarly $a\succ b$ if $a = b+1$. Put $l=n-k$.

    Assume that $\lambda$ corresponds to an Ulrich bundle. Then
    \begin{equation}
        \label{eq:igr_pf_1}
        \Irr(\lambda)=\{1,2,\ldots, d\}.
    \end{equation}
    Put $D=(d-1)/2$ and
    \begin{equation}
        \label{eq:alpha_pm_ij}
        \alpha^\pm_{ij} = \alpha_i \pm \beta_j - \frac{d+1}{2},
        \qquad
        \alpha_{ij} = \frac{\alpha_i+\alpha_j}{2} - \frac{d+1}{2}.
    \end{equation}
    One immediately observes that
    \begin{equation}
        \label{eq:alpha_pm}
        \alpha^\pm_{ij}=\alpha_{ii}\pm\beta_j,\qquad \alpha_{ij} =
            \frac{\alpha_{ii}+\alpha_{jj}}{2}\quad\text{for $1\leq i<j\leq k$.}
    \end{equation}

    It follows from~\eqref{eq:igr_pf_1} that the values $\alpha^\pm_{ij}$
    and $\alpha_{ij}$ are distinct and fill the integer range $[D, \ldots, -D]$.
    Also, from~\eqref{eq:ab_dom} we get inequalities
    \begin{alignat}{3}
        \notag
        \alpha_{ij} & \leq  \alpha_{pq} &
                    & &
        \qquad  & \text{if $i\geq p$ and $j\geq q$}, \\
        \label{eq:a_pm_ij_ord}
        \alpha^+_{i1} & > \alpha^+_{i2} > \ldots > \alpha^+_{il} >{} &
        \alpha_{ii} > \alpha^-_{il} > \ldots > \alpha^-_{i2} & > \alpha^-_{i1} &
                                                             & \text{for all $1\leq i\leq k$}, \\
        \notag
        \alpha^+_{1j} & > \alpha^+_{2j} > \ldots > \alpha^+_{kj}, &
        \alpha^-_{kj}  < \ldots < \alpha^-_{2j} & < \alpha^-_{1j} &
                                                & \text{for all $1\leq j\leq l$}.
    \end{alignat}
    From the proof of Lemma~\ref{lm:igr_dim} and~\eqref{eq:alpha_pm_ij} we see that
    \begin{equation}
        \label{eq:aij_val}
        \alpha^+_{11} = D,\quad\alpha^-_{k1} = -D,\quad \alpha_{11}=-\alpha_{kk},\quad
        \alpha_{1k}=0.
    \end{equation}

    Let us first assume that $(\beta_1, \beta_2,\ldots, \beta_l)\neq (l, l-1,\ldots, 1)$.
    The last condition is equivalent to the existence of such a $t\in\{1,\ldots,l\}$ that
    \begin{equation*}
        \beta_1 \succ \beta_2 \succ \ldots \succ \beta_t > 1.
    \end{equation*}
    It follows from~\eqref{eq:alpha_pm} that
    \begin{equation*}
        D = \alpha^+_{11} \succ \alpha^+_{12}\succ \ldots \succ \alpha^+_{1t} \not\succ \alpha_{11}.
    \end{equation*}
    From inequalities~\ref{eq:a_pm_ij_ord} we deduce that the next position after $\alpha^+_{1t}$
    must be taken by $\alpha^+_{21}$. In particular,
    \begin{equation}
        \label{eq:a21}
        \alpha^+_{21} = \alpha_{22} + \beta_1 = \alpha^+_{11}-t.
    \end{equation}
    A similar argument shows that
    \begin{equation*}
        -D = \alpha^-_{k1} \prec \alpha^-_{k2} \prec \ldots \prec \alpha^-_{kt} \not\prec \alpha^-_{k-1,1},
    \end{equation*}
    which implies
    \begin{equation}
        \label{eq:ak11}
        \alpha^-_{k-1,1} = \alpha_{k-1,k-1} - \beta_1 = \alpha^-_{11}+t.
    \end{equation}
    Combining~\eqref{eq:a21} and~\ref{eq:ak11}, we deduce
    \begin{equation*}
        \alpha_{2,k-1} = \frac{\alpha_{22}+\alpha_{k-1,k-1}}{2} = \frac{\alpha_{11}+\alpha_{kk}}{2} = \alpha_{1k}.
    \end{equation*}
    That contradicts the assumption that all the values $\alpha_{ij}$ are distinct.

    Finally, in case $(\beta_1, \beta_2,\ldots, \beta_l)=(l,l-1,\ldots,1)$ one has
    \begin{equation*}
        D=\alpha^+_{11} \succ \alpha^+_{12} \succ \ldots \succ \alpha^+_{1k} \succ
        \alpha_{11} \succ \alpha^-_{1k} \succ \ldots \succ \alpha^-_{12} \succ \alpha^-_{11}
        > \alpha_{12} > \alpha^+_{21} \succ \alpha^+_{22} \succ \ldots \succ \alpha^+_{2k}
        \succ \alpha_{22}.
    \end{equation*}
    Inequalities~\eqref{eq:a_pm_ij_ord} imply that $\alpha_{12}$ is the only possible
    immediate predecessor of $\alpha^-_{11}$, thus
    \begin{equation*}
        \alpha_{12} = \frac{\alpha_{11}+\alpha_{22}}{2} = \alpha_{11}-(l+1)
        \quad \Rightarrow \quad
        \alpha_{22} = \alpha_{11}-2(l+1).
    \end{equation*}
    A similar argument shows that
    \begin{equation*}
        \alpha_{k-1,k} = \frac{\alpha_{k-1,k-1}+\alpha_{kk}}{2} = \alpha_{kk}+(l+1)
        \quad \Rightarrow \quad
        \alpha_{k-1,k-1} = \alpha_{kk}+2(l+1).
    \end{equation*}
    Combining the last two equalities we deduce that
    \begin{equation*}
        \alpha_{2,k-1} = \frac{\alpha_{22}+\alpha_{k-1,k-1}}{2} = \frac{\alpha_{11}+\alpha_{kk}}{2} = \alpha_{11},
    \end{equation*}
    which once again contradicts the assumption that all $\alpha_{ij}$ are distinct.
\end{proof}

\subsection{Isotropic Grassmannians of planes}
\label{ssec:igr_2}
We would like to classify all irreducible equivariant
Ulrich bundles on isotropic Grassmannians of planes. Thus, throughout
this section $k=2$ and $X=\IGr(2,V)$.

\begin{proposition}
    \label{prop:igr2}
    Let $p$ be a positive divisor of $n-1$ such that $(n-1)/p=2q+1$ is odd. Then
    the bundle on $\IGr(2,V)$ corresponding to the weight
    \begin{equation}
        \label{eq:igr2l}
        \lambda = (n-2+p,\: n-1-p,
            \: \underbrace{2pq,\: \ldots,\: 2pq}_{2p},
            \: \underbrace{2p(q-1),\: \dots,\: 2p(q-1)}_{2p},
            \: \ldots,
            \: \underbrace{2p,\: \ldots,\: 2p}_{2p},
            \: \underbrace{0,\: \ldots,\: 0}_{p-1}
        )
    \end{equation}
    is Ulrich. Moreover, all irreducible equivariant Ulrich bundles on $\IGr(2,V)$ are of this form.
\end{proposition}
\begin{proof}
    We will freely use the notation introduced along the proof of Proposition~\ref{prop:igr}.
    Remark that the only $\alpha_{ij}$ we are left with are $\alpha_{12}=0$ and
    $\alpha_{11}=-\alpha_{22}$. Let us also spell out $d=4n-5$, $l=n-2$, and $D=2n-3$.

    Assume that the values $\alpha^\pm_{ij}$ and $\alpha_{ij}$ cover the range
    $[D,\ldots,-D]$, and put $p=\alpha_{11}$. Let us show by induction on $l$
    that $l=2pq+(p-1)$ for some integer $q\geq 0$ and that
    \begin{align}
        \notag
        \beta_1 &= 2n-3-p     & \beta_{2p+1} &= 2n-3-5p  & &\ldots & \beta_{2p(q-1)+1} &= 5p-1 & \beta_{2pq+1} &= p-1  \\
        \label{eq:igr2bi}
        \beta_2 &= 2n-4-p     & \beta_{2p+2} &= 2n-4-5p  & &\ldots &  \beta_{2p(q-1)+2} &= 5p-2 & & \ \:\vdots \\
        \notag
        & {}\ \:\vdots      &              & {}\ \:\vdots    & &       &  & {}\ \:\vdots & \beta_{l}&=1 \\
        \notag
        \beta_{2p} &= 2n-2-3p & \beta_{4p}   &= 2n-2-7p  & &\ldots & \beta_{2p(q-1)+2p} &= 3p
    \end{align}
    As $l=n-2$, the equality $l=2pq+(p-1)$ is equivalent to $n-1=p(2q+1)$, which is what
    we would like to show. Also, by~\eqref{eq:igr_ab} and~\eqref{eq:alpha_pm_ij} one has
    \[
        \lambda_1 = \alpha_{11} + \frac{d+1}{2}-n,\qquad
        \lambda_2 = \alpha_{22} + \frac{d+1}{2}-(n-1),\qquad
        \lambda_i = \beta_{i-2} - (n+1-i)\quad \text{for $2<i\leq n$.}
    \]
    Simplifying these formulas, one comes to a conclusion that~\eqref{eq:igr2bi}
    corresponds exactly to~\eqref{eq:igr2l}. Thus, we are indeed about to prove
    an equivalent statement.

    In the base case $l=3$ we need to find all possible integers $p, \beta_1>0$
    such that $\pm p\pm \beta_1$ and $\pm p$ are distinct and take all the values
    $\{\pm 1, \pm 2, \pm 3\}$. One immediately checks that the only possibility is
    $p=2$ and $\beta_1=1$.

    Let us prove the inductive step.
    Recall that the $\alpha^+_{11}=D$ is necessarily the maximal element.
    From~\eqref{eq:a_pm_ij_ord} we know that the only elements that can take the
    values $\{\pm 1, \pm 2, \dots, \pm (p-1)\}$ are $\alpha^\pm_{ij}$. Thus, by
    the pigeonhole principle $l \geq p-1$.
    Next, if $l=p-1$, then by the same principle $\alpha^-_{21}<p$. As a consequence,
    \[
        D=\alpha^+_{11} \succ \alpha^+_{12}\succ \ldots \succ \alpha^+_{1,p-1}\succ \alpha_{11}=p,
    \]
    which implies
    \[
        \beta_1\succ \beta_2\succ \ldots \succ \beta_{p-1}=1.
    \]
    That is exactly what~\eqref{eq:igr2bi} gets reduced to.

    Finally, assume that $l>p-1$. By the pigeonhole principle again
    and the order on $\alpha^-_{2,i}$ we argue that 
    $\alpha^+_{21}=\alpha^+_{11}-2p>\alpha_{11}$. Looking at~\eqref{eq:a_pm_ij_ord},
    we deduce that
    \[
        D=\alpha^+_{11}\succ \alpha^+_{12}\succ \ldots \succ \alpha^+_{1,2p}\succ \alpha^+_{21}.
    \]
    In particular,
    \[
        \beta_1\succ \beta_2\succ \ldots \succ \beta_{2p} \not\succ \beta_{2p+1},
    \]
    which implies
    \begin{equation}
        \label{eq:igr2b}
        \begin{split}
            D={}&\alpha^+_{11} \succ \alpha^+_{12}\succ \ldots \succ \alpha^+_{1,2p}\succ
            \alpha^+_{21}\succ \alpha^+_{22}\succ \ldots \succ \alpha^+_{2,2p},\\    
            &\alpha^-_{1,2p} \succ \ldots \succ \alpha^-_{12}\succ \alpha^-_{11}\succ
            \alpha^-_{2,2p}\succ \ldots \succ \alpha^-_{22}\succ \alpha^-_{21} = -D.
        \end{split}
    \end{equation}
    We conclude that
    \begin{equation}
        \label{eq:igr2bf}
        (\beta_1,\beta_2,\ldots,\beta_{2p}) = (D-p, D-p-1,\ldots,D-3p+1).
    \end{equation}

    In follows from~\eqref{eq:igr2b} that the range $[D-4p,\ldots,-D+4p]$ is covered
    by the values
    \[
        \alpha_{11}=p,
        \quad \alpha_{12}=0,
        \quad \alpha_{22}=-p,
        \quad \text{and}
        \quad \left\{\alpha^\pm_{ij}\right\}_{2p < j \leq l}.
    \]
    Applying the induction hypothesis for $l'=l-2p$ we find that $l-2p=2pq'+(p-1)$.
    In particular, $l=2pq+(p-1)$ for $q=q'+1$.
    We also get the exact values of $\beta_j$ for $j>2p$ in terms of $p$ and $n'=n-2p$.
    Rewriting these values back in terms of $p$ and $n$ together with~\eqref{eq:igr2bf}
    proves~\eqref{eq:igr2bi}.
\end{proof}

Remark that for any $n$ one has an obvious choice $p=n-1$. Then
$(n-1)/p=1$ is odd and the corresponding weight is $(2n-3,0,\ldots,0)$.
Moreover, if $n-1$ has no odd divisors apart from $1$, this choice of $p$
is the only possible. Summarising, we get the following statement.

\begin{corollary}
    For any $n$ the bundle $S^{2n-3}\CU^*$ on $\IGr(2, 2n)$ is Ulrich. 
    If $n=2^r+1$, it is the only irreducible equivariant Ulrich bundle
    up to isomorphism.
\end{corollary}

\begin{example} The following are the smallest isotropic Grassmannians of planes
    that admit two and three nonisomorphic irreducible Ulrich bundles respectively.
    \begin{enumerate}
        \item There are exactly two nonisomorphic equivariant irreducible
            Ulrich bundles on $\IGr(2,8)$:
            \[
                S^5\CU^*
                \quad \text{and}
                \quad \CU^*(2)\otimes \Sigma^{[2,2]}(\CU^\perp\!/\:\CU).
            \]
        \item There are exactly three nonisomorphic equivariant irreducible
            Ulrich bundles on $\IGr(2,20)$:
            \[
                S^{17}\CU^*,
                \quad S^5\CU^*(6)\otimes \Sigma^{[6,6,6,6,6,6]}(\CU^\perp\!/\:\CU),
                \quad \text{and}
                \quad \CU^*(8)\otimes \Sigma^{[8,8,6,6,4,4,2,2]}(\CU^\perp\!/\:\CU).
            \]
    \end{enumerate}
    Here $\Sigma^{[\mu]}$ denotes the symplectic Schur functor corresponding to $\mu$.
\end{example}

\section{Odd orthogonal Grassmannians}
\label{sec:ogr}
We would like to deal now with type $B_n$. Namely, let $\sG=\SO(V)$,
where $V$ is a $(2n+1)$-dimensional vector space equipped with a non-degenerate
symmetric form. We identify the root system of $\sG$ with
the set of all integer vectors of length $\sqrt{2}$ or $1$. For the simple roots
we choose $\eps_i=e_i-e_{i+1}$ for $i=1,\ldots,n-1$, and $\eps_n=e_n$. The fundamental
weights then are $\omega_i=e_1+\ldots+e_i$ for $i=1,\ldots,n-1$ and
$\omega_n=\frac{1}{2}(e_1+\ldots+e_n)$, while
\[
    \rho=\left(n-\tfrac{1}{2},\: n-\tfrac{3}{2},\:\ldots,\:\tfrac{1}{2}\right).
\]

We would like to classify equivariant irreducible vector bundles on $X=\sG/\sP_k\simeq\OGr(k,V)$
for all $k=1,\ldots,n-1$.
The dimension of $X$ is $d=k(2n-k)-k(k+1)/2$.
The weight lattice $P_\sG$ is identified with the set of vectors with all integer or
halfinteger coordinates,
$P_\sG=\ZZ^n\cup (\thalf+\ZZ)^n$,
while the dominant cone $P^+_\sL$ consists of all $\lambda\in P_\sG$ such that
\[
    \lambda_1\geq\lambda_2\geq\ldots\geq\lambda_k,\qquad
    \lambda_{k+1}\geq\ldots\geq\lambda_n\geq 0.
\]

Let us again introduce the notation
\[
    (\alpha_1,\alpha_2,\ldots,\alpha_k, \beta_1,\beta_2,\ldots,\beta_{n-k})
= \lambda + \rho
\]
For $\lambda\in P^+_\sL$ we deduce inequalities
\begin{equation}
    \label{eq:b:abdom}
    \alpha_1>\alpha_2>\ldots>\alpha_k,\qquad \beta_1>\beta_2>\ldots>\beta_{n-k}>0.
\end{equation}
Remark that the only difference between~\eqref{eq:b:abdom} and~\eqref{eq:ab_dom}
is that in type $B_n$ the terms $\alpha_i$ and $\beta_j$ are allowed to be
halfinteger all at once.

The root systems in type $B_n$ and $C_n$ are extremely similar. An argument analagous
to the one used in Proposition~\ref{prop:igr:irr} shows that
\[
    \Irr(\lambda)=\{\alpha_i\pm\beta_j\}_{1\leq i\leq k,\ 1\leq j\leq n-k}\ \bigcup
    \ \ZZ\cap\left\{\frac{\alpha_i+\alpha_j}{2}\right\}_{1\leq i\leq j\leq k}.
\]
As the cardinality of $\Irr(\lambda)$ is at most $d$, we deduce that all the values
$(\alpha_i+\alpha_j)/2$ are integer. In particular, all $\alpha_i$ are integer, thus
$\beta_j$ are forced to be as well. We come to the following conclusion.

\begin{proposition}
    A weight $\lambda=(\lambda_1,\lambda_2,\ldots,\lambda_n)$ corresponds to an Ulrich
    bundle on $\OGr(k, 2n+1)$ if and only if 
    $(\lambda_1-\tfrac{1}{2}, \lambda_2-\tfrac{1}{2},\ldots,\lambda_n-\frac{1}{2})$
    is a weight corresponding to an Ulrich bundle for $\IGr(k, 2n)$.
\end{proposition}
\begin{proof}
    The conditions on $\alpha$ and $\beta$ under which the bundle $\CU^\lambda$ is Ulrich
    are exactly the same in types $C_n$ and $B_n$.
    Thus, $\lambda$~is Ulrich for $\OGr(k,2n+1)$ if and only if
    $\lambda+\rho_{B_n}-\rho_{C_n}$ is Ulrich for $\IGr(k,2n)$.
\end{proof}

Using Proposition~\ref{prop:igr2} we can now find all
equivariant Ulrich bundles on $\OGr(2,2n+1)$.
\begin{corollary}
    Every irreducible equivariant Ulrich bundle $\CU^\lambda$ on $\OGr(2,2n+1)$
    corresponds to $\lambda$ of the form
    \begin{equation}
        \label{eq:b:og2l}
        \lambda = (n-\tfrac{3}{2}+p,\: n-\tfrac{1}{2}-p,
            \: \underbrace{2pq+\tfrac{1}{2},\: \ldots,\: 2pq+\tfrac{1}{2}}_{2p},
            \: \ldots,
            \: \underbrace{2p+\tfrac{1}{2},\: \ldots,\: 2p+\tfrac{1}{2}}_{2p},
            \: \underbrace{\tfrac{1}{2}\: \ldots,\: \tfrac{1}{2}}_{p-1}
        ),
    \end{equation}
    where $p$ is a positive divisor of $n-1$ such that $(n-1)/p=2q+1$ is odd.
\end{corollary}

Let $\CS$ denote the \emph{spinor} bundle on $X$, namely,
\[
    \CS = \CU^{(\thalf,\: \thalf,\:\ldots,\: \thalf)}.
\]
\begin{corollary}
    Up to isomorphism $\CS$ is the only irreducible equivariant Ulrich bundle
    on an odd-dimensional quadric $Q^{2n-1}$.
\end{corollary}
\begin{proof}
    The only irreducible Ulrich bundle on $\PP^{2n-1}$ is $\CO$, which corresponds
    to $\lambda=(0,0,\ldots,0)$.
\end{proof}

\begin{corollary}
    \label{cor:bn}
    The only nonmaximal odd orthogonal Grassmannians that admit an irreducible
    equivariant Ulrich bundle
    are the odd-dimensional quadrics $\OGr(1, V)\simeq Q^{2n-1}$ and the Grassmannians
    of isotropic planes $\OGr(2, V)$.
\end{corollary}

\section{Even orthogonal Grassmannians}
\label{sec:dn}
Let us now deal with groups of type $D_n$. Namely, let $\sG=\SO(V)$, where
$V$ is a $2n$-dimensional vector space equipped with a non-degenerate
symmetric form.
We identify the root of $\sG$ system with the set of vectors $\pm (e_i\pm e_j)$, $i\neq j$.
The simple roots are $\eps_i=e_i-e_{i+1}$ for $i=1,\ldots, n-1$ and
$\eps_n=e_{n-1}+e_{n}$,
while the fundamental weights are $\omega_i=e_1+\ldots+e_i$ for $i=1,\ldots,n-2$, and
\[
    \begin{aligned}
        \omega_{n-1} &= \thalf(e_1+\ldots+e_{n-1}-e_n), \\
        \omega_n     &= \thalf(e_1+\ldots+e_{n-1}+e_n).
    \end{aligned}
\]
Thus, $\rho=(n-1, n-2, \ldots, 0)$.

Let $X=\sG/\sP_k$. The dimension of $X$ is $d=k(2n-k)-k(k+1)/2$.
It is well known that $X\simeq \OGr(k, V)$ for $k=1,\ldots, n-2$,
while $\sG/\sP_{n-1}$ and $\sG/\sP_{n}$ are isomorphic to the two connected components
of the Grassmannian of maximal isotropic subspaces $\OGr(n, V)$.
The latter two varieties will
be treated in Section~\ref{sec:max}, here we assume $1\leq k\leq n-2$.

As in type $B_n$, the weight lattice $P_\sG$ is identified with the set of all
(half)integer vectors $\ZZ^n\cup(\thalf+\ZZ)^n$,
while the dominant cone $P^+_\sL$ consists of all $\lambda\in P_\sG$ such that
\[
    \lambda_1\geq\lambda_2\geq\ldots\geq\lambda_k,\qquad
    \lambda_{k+1}\geq\ldots\geq|\lambda_n|.
\]

We once again introduce the notation
\[
    (\alpha_1,\alpha_2,\ldots,\alpha_k, \beta_1,\beta_2,\ldots,\beta_{n-k})
= \lambda + \rho
\]
For $\lambda\in P^+_\sL$ one has inequalities
\begin{equation}
    \label{eq:d:abdom}
    \alpha_1>\alpha_2>\ldots>\alpha_k,\qquad \beta_1>\beta_2>\ldots>|\beta_{n-k}|.
\end{equation}

An argument similar to the one used Proposition~\ref{prop:igr:irr} shows that
\begin{equation}
    \label{eq:d:irr}
    \Irr(\lambda)=\{\alpha_i\pm\beta_j\}_{1\leq i\leq k,\ 1\leq j\leq n-k}
    \ \bigcup
    \ \ZZ\cap\left\{\frac{\alpha_i+\alpha_j}{2}\right\}_{1\leq i<j\leq k}.
\end{equation}

\begin{remark}
    \label{rem:d}
    Equality~\eqref{eq:d:irr} together with $\beta_{n-k}=\lambda_n$ imply that $\lambda$
    corresponds to an Ulrich bundle if and only if the weight
    $\lambda'=(\lambda_1,\lambda_2,\ldots,\lambda_{n-1},-\lambda_n)$ does.
    Moreover, $\CU^\lambda$ can not be Ulrich as long as $\lambda_n=0$.
    Thus, it is sufficient to classify irreducible equivariant Ulrich bundles with
    $\lambda_n > 0$.
\end{remark}

\subsection{General orthogonal Grassmannians}
\label{ssec:ogr:big}
We are first going to prove the following statement.

\begin{proposition}
    \label{prop:d}
    There are no irreducible equivariant Ulrich bundles on $\OGr(k,V)$ for $3<k\leq n-2$.
\end{proposition}
\begin{proof}
    The proof is very similar to the proof of Proposition~\ref{prop:igr}.
    Put $l=n-k$ and write $a\prec b$ ($a\succ b$) whenever $a+1=b$ ($a=b+1$ respectively).
    Let us assume that $\Irr(\lambda)=\{1,\ldots,d\}$. Then necessarily 
    \begin{equation*}
        \alpha_1+\beta_1 = d,\quad \alpha_k-\beta_1 = 1,\quad \alpha_1+\alpha_k=d+1.
    \end{equation*}
    To keep notation consistent, we introduce $\alpha^\pm_{ij}=\alpha_i\pm\beta_j$ and
    $\alpha_{ij}=\tfrac{\alpha_i+\alpha_j}{2} - \tfrac{d+1}{2}$. According to our assumption,
    the values $\alpha^\pm_{ij}$ together with $\alpha_{ij}$ for $1\leq i<j\leq k$ are all
    integer, distinct and span the range $[D,\ldots,-D]$, where $D=(d-1)/2$.
    According to Remark~\ref{rem:d}, we can also restrict ourselves to the case
    $\beta_l>0$, thus inequalities~\eqref{eq:a_pm_ij_ord} remain valid.

    In case $\beta=(l-\thalf, \ldots, \frac{3}{2}, \thalf)$ one has
    \begin{equation*}
        \alpha^+_{i1}\succ \alpha^+_{i2}\succ \cdots\succ \alpha^+_{il}\succ
        \alpha^-_{il}\succ \cdots \succ \alpha^-_{i2}\succ \alpha^-_{i1}
    \end{equation*}
    for all $i=1,\ldots, k$. In particular, $\alpha^-_{11}\succ\alpha_{12}$ and
    $\alpha^+_{k,1}\prec \alpha_{k-1,k}$, which implies that $\alpha_{2,k-1}=\alpha_{1,k}$.

    In case $\beta\neq(l-\thalf, \ldots, \frac{3}{2}, \thalf)$ we claim that
    \begin{equation}
        \label{eq:ogr:s2}
        \alpha^+_{11}\succ \alpha^+_{12}\succ \cdots\succ \alpha^+_{1t}\succ\alpha^+_{21}
    \end{equation}
    for some $1\leq p\leq k$. Indeed, otherwise
    \[
        \alpha^+_{11}\succ \alpha^+_{12}\succ \cdots\succ \alpha^+_{1l}\succ\alpha_{12}.
    \]
    In particular, $\alpha_1+\beta_1=\frac{\alpha_1+\alpha_2}{2}+1$. As $\beta_1\geq\thalf$
    and $\alpha_2\leq \alpha_1-1$, we deduce that $\beta=(l-\thalf, \ldots, \frac{3}{2}, \thalf)$.
    Similarly, we argue that
    \begin{equation*}
        \alpha^-_{k,1}\prec \alpha^-_{k,2}\prec \cdots\prec \alpha^-_{k,t}\prec\alpha^-_{k-1,1}
    \end{equation*}
    for the same $p$ as in~\eqref{eq:ogr:s2}, which implies that
    \[
        \alpha_2=\alpha_1-p,\quad\text{and}\quad \alpha_{k-1} = \alpha_k+p.
    \]
    Thus, $\alpha_{2,k-1}=\alpha_{1,k}$, which contradicts the assumption that all the values
    are distinct.
\end{proof}

\subsection{Small even orthogonal Grassmannians}
\label{ssec:d:small}
Let us classify irreducible equivariant Ulrich bundles on $\OGr(k,2n)$ for
$k \leq 3$. We will freely use the notation introduced along the proof of Proposition~\ref{prop:d}.

\subsubsection{Even-dimensional quadrics}
\begin{proposition}
    \label{prop:d:q}
    The only irreducible equivariant Ulrich bundles on an even-dimensional quadric $Q^{2n-2}$
    up to isomorphism are the spinor bundles $\CS^\pm$, which correspond to the weights
    \(
        \lambda^\pm = (\thalf, \thalf,\ldots,\thalf,\pm\thalf).
    \)
\end{proposition}
\begin{proof}
    According to inequalities~\eqref{eq:a_pm_ij_ord}, the only possible complete ordering is
    \[
        2n-2=\alpha_1+\beta_1\succ\cdots\succ\alpha_1+|\beta_{n-1}|\succ
        \alpha_1-|\beta_{n-1}|\succ\cdots\succ\alpha_1-\beta_1=1.
    \]
    Thus, $(\alpha, \beta_1,\ldots,\beta_{n-2},|\beta_{n-1}|)=(n-\thalf,n-\tfrac{3}{2},\ldots,\tfrac{3}{2},\thalf)$.
    It remains to subtract $\rho=(n-1,\ldots,1,0)$.
\end{proof}

\subsubsection{Even orthogonal Grassmannians of planes}
\begin{proposition}
    \label{prop:d:2}
    Let $p\in\thalf+\ZZ$ and $q\in\ZZ$ be nonnegative integers such that $n-2=p(2q+1)-\thalf$. Then
    the bundle on $X=\OGr(2,2n)$ corresponding to the weight
    \begin{multline}
        \label{eq:d:u2}
        \lambda = (n-2+p,\: n-1-p,
            \: \underbrace{2pq+\thalf,\: \ldots,\: 2pq+\thalf}_{2p},
            \: \underbrace{2p(q-1)+\thalf,\: \dots,\: 2p(q-1)+\thalf}_{2p},
            \: \ldots,\\
            \: \underbrace{2p+\thalf,\: \ldots,\: 2p+\thalf}_{2p},
            \: \underbrace{\thalf,\: \ldots,\: \thalf}_{p-\thalf}
        )
    \end{multline}
    is Ulrich. Moreover, all irreducible equivariant Ulrich bundles on $X$ are of this form
    up to negating $\lambda_n$.
\end{proposition}
\begin{proof}
    The situation is very similar to that in type $C_n$.
    According to Remark~\ref{rem:d}, we may assume that $\lambda_n>0$.
    The dimension of $X$ equals $d=4n-7$, thus $D=2n-4$.
    Let $\alpha_{1,1}=-\alpha_{2,2}=p$. It is enough to show that
    if the values $\alpha^\pm_{i,j}$ and $\alpha_{1,2}=0$ cover the whole
    integer range $[2n+4,\ldots,-2n+4]$, then
    \begin{align}
        \notag
        \beta_1 &= 2n-4-p     & \beta_{2p+1} &= 2n-4-5p  & &\ldots & \beta_{2p(q-1)+1} &= 5p-1 & \beta_{2pq+1} &= p-1  \\
        \label{eq:d:2b}
        \beta_2 &= 2n-5-p     & \beta_{2p+2} &= 2n-5-5p  & &\ldots &  \beta_{2p(q-1)+2} &= 5p-2 & & \ \:\vdots \\
        \notag
        & {}\ \:\vdots      &              & {}\ \:\vdots    & &       &  & {}\ \:\vdots & \beta_{n-2}&=\thalf \\
        \notag
        \beta_{2p} &= 2n-3-3p & \beta_{4p}   &= 2n-3-7p  & &\ldots & \beta_{2p(q-1)+2p} &= 3p
    \end{align}
    As in the proof of Proposition~\ref{prop:igr2}, we show the latter by induction on $l$.
    For the base case we take $\alpha^+_{2,1}=-p+\beta_1<p$. Then, by pigeonhole principle
    $2l=2p-1$ and $\beta=(p-1,\ldots,\tfrac{3}{2},\thalf)$.  
    Otherwise,
    \[
        \alpha^+_{1,1}\succ\alpha^+_{1,2}\succ\cdots\succ\alpha^+_{1,t}\succ\alpha^+_{2,t}
    \]
    for some $1\leq t\leq l$, from which we deduce that $t=2p$ and
    \[
        D-p=\beta_1\succ\beta_2\succ\cdots\succ\beta_{2p},
    \]
    and reduce the statement to $l'=l-2p$.
\end{proof}

\begin{example}
    For any even orthogonal Grassmannian of planes $\OGr(2, 2n)$ there are at least four
    nonisomorphic irreducible equivariant Ulrich vector bundles. In terms of Proposition~\ref{prop:d:2}
    they correspond to the values $p=n-\tfrac{3}{2}$, $q=0$, and $p=\thalf$, $q=n-2$ respectively.
    The former values produce the bundles $S^{2n-4}\CU^*\otimes\CS^\pm$, where
    $S^{2n-4}\CU^*$ is the $(2n-4)$-th symmetric power of the dual tautological rank~$2$ bundle,
    and $\CS^\pm$ are the spinor bundles. In the latter case
    \[
        \lambda=(n-\tfrac{3}{2},\:n-\tfrac{3}{2},\:n-\tfrac{3}{2},
        \:n-\tfrac{5}{2},\:\ldots,\:\tfrac{5}{2},\:\pm\tfrac{3}{2}).
    \]
\end{example}

\subsubsection{Even orthogonal Grassmannians of $3$-spaces}
We finally turn to the case $k=3$.

\begin{proposition}
    \label{prop:d:3}
    The Grassmannian $\OGr(3,2n)$ carries an irreducible equivariant Ulrich bundle
    if and only if $n-3=2q$. In the latter case the only equivariant irreducible
    Ulrich bundles are isomorphic to $\CU^{\lambda^\pm}$ for
    \begin{equation}
        \label{eq:d:3l}
        \lambda^\pm=(2n-4,2n-5,2n-6,2n-6,2n-6,2n-10,2n-10,\ldots, 4, \pm 4).
    \end{equation}
\end{proposition}
\begin{proof}
    According to Remark~\ref{rem:d},
    we may assume that $\lambda_n=\beta_l>0$. Let $\CU^\lambda$ be Ulrich.
    Following the proof of Proposition~\ref{prop:d} we see that
    $\alpha_{1,1}=-\alpha_{3,3}$ and $\alpha_{1,2}=\alpha_{2,3}$, which
    implies $\alpha_{2,2}=0$. Denote $p=\alpha_{1,1}$. Remark that
    $p\in 2\ZZ$ as $\alpha_{1,2}=p/2\in\ZZ$. Let us show
    that $\beta_1\geq p$. Indeed, otherwise from inequalities
    \[
        p>\beta_1>\beta_2>\ldots>\beta_l>0
    \]
    we deduce from~\eqref{eq:a_pm_ij_ord} that none of the $\alpha^\pm_{ij}$
    can be equal to $p\in[D,\ldots,-D]$.
    Now, we argue by induction that $l=2q$ is even and
    \begin{equation}
        \label{eq:d:3}
        \beta=(D-2,D-3,D-6,D-7,\ldots,5,4),\quad p=2.
    \end{equation}
    The base cases are $l=0$ and $l=1$. In the former 
    $p$ trivially equals $2$, while in the latter one can show with
    brute force that $\lambda$ does not exist.

    Now, for $l>1$ we have $\alpha^+_{2,1}\geq p$, and from~\eqref{eq:a_pm_ij_ord} one has
    \[
        \alpha^+_{1,1}\succ\alpha^+_{1,2}\succ\cdots\succ\alpha^+_{1,t}\succ\alpha^+_{2,1}
    \]
    for some $1\leq t\leq l$.
    We immediately conclude that $t=p$ and
    \[
        D-p=\beta_1\succ\beta_2\succ\cdots\succ\beta_p.
    \]
    In remains to observe that $\alpha^+_{3,t}=\alpha^+_{3,1}-2p$, thus
    \[
        \alpha^+_{1,1}\succ\cdots\succ\alpha^+_{1,p}\succ
        \alpha^+_{2,1}\succ\cdots\succ\alpha^+_{2,p}\succ
        \alpha^+_{3,1}\succ\cdots\succ\alpha^+_{3,p},
    \]
    from which we see that by dropping $(\beta_1,\ldots,\beta_p)$ the problem gets reduced
    to $l'=l-p$, and that~\eqref{eq:d:3} is equivalent to~\eqref{eq:d:3l}, as $D=3n-8$.
\end{proof}

\section{Maximal isotropic Grassmannians}
\label{sec:max}

\subsection{Lagrangian Grassmannians}
\label{ssec:lgr}
Let $V$ be a $2n$-dimensional symplectic vector space and
let $X=\LGr(n ,V)$ be the Lagrangian Grassmannian of maximal
isotropic subspaces.  Our goal is to prove the following.

\begin{proposition}
    \label{prop:c:max}
    The only Lagrangian Grassmannian possessing an irreducible equivariant Ulrich bundle
    is $\LGr(2, 4)$. The only irreducible equivariant Ulrich bundle on $\LGr(2,4)$
    up to isomorphism is $\CU^*$.
\end{proposition}

In the Lagrangian case Proposition~\ref{prop:igr:irr} simplifies to
\[
    \Irr(\lambda) = \left\{\frac{\alpha_i+\alpha_j}{2}\right\}_{1\leq i\leq j\leq n.}
\]
Recall that $\CU^\lambda$ is Ulrich if and only if $\Irr(\lambda)=\{1,2,\ldots,d\}$.
Informally speaking, what we need is to find all integer sequences of length $n$
such that the half sums fill in all the blanks between the elements without any collisions.

Consider a partially ordered set $A=\{a_{ij}\}_{1\leq i\leq j}$ with elements, indexed
by unordered pairs of positive integers, subject to the partial order
\begin{equation}
    \label{eq:lg_a_ord}
    a_{ij} \leq a_{pq}\quad \text{if $i\leq p$ and $j\leq q$}.
\end{equation}
Denote by $A_n$ the subset
\begin{equation}
    \label{eq:lg_a_n}
    A_n=\{a_{ij}\}_{1\leq i\leq j \leq l}\subset A.
\end{equation}
In particular, there is an increasing chain of inclusions
$\emptyset = A_0\subset A_1\subset \ldots \subset A_n\subset \ldots\subset A$.

\begin{lemma}
    The partially ordered set $A\setminus A_n$ has a least element
    \[
        \min(A\setminus A_n) = a_{1, n+1}.
    \]
\end{lemma}
\begin{proof}
    Immediately follows from~\eqref{eq:lg_a_ord}.
\end{proof}

\begin{proposition}
    \label{prop:lg_tij}
    There exists a unique order-preserving mapping $T:A\to\ZZ_+$
    from $A$ to the set of positive integers $\ZZ_+$ such that
    \begin{align}
        \label{eq:lg_tij_mincond}
        T(\min(A\setminus A_n)) & = \min(\ZZ_+\!\setminus T(A_n)) & \text{for all $n>0$}, \\
        \label{eq:lg_tij}
        T(a_{ij}) & = \frac{T(a_{ii})+T(a_{jj})}{2} & \text{for all $1\leq i<j$}.
    \end{align}
    If we denote $\alpha_{ij}=T(a_{ij})$, then $\alpha_{11}=1$, $\alpha_{ii}=4i-5$ for all $i>1$.
\end{proposition}
\begin{proof}
    We start with an observation that condition~\eqref{eq:lg_tij} is consistent with the partial order
    given by~\eqref{eq:lg_a_ord}.
    Also, the equality $\alpha_{11}=1$ obviously follows from~\eqref{eq:lg_tij_mincond} and $A_0=\emptyset$.
    Let us prove by induction that for all $n>1$
    \begin{equation}
        \label{eq:lg_prop_ind}
        \alpha_{nn} = 4n-5, \qquad
        \{1,2,\ldots,2n-1\}\subseteq T(A_n), \qquad
        2n \notin T(A_n).
    \end{equation}
    The base case, $n=2$, easily follows from~\eqref{eq:lg_tij_mincond}.
    Condition~\eqref{eq:lg_tij} implies $\alpha_{11}=1$, $\alpha_{12}=2$, $\alpha_{22}=3$.

    Assume that~\eqref{eq:lg_prop_ind} holds for all $l=2,\ldots,n$.
    As $2n\notin T(A_n)$ and all the smaller positive integers are contained in $T(A_n)$
    by inductive hypothesis, condition~\eqref{eq:lg_tij_mincond} dictates
    $\alpha_{1,n+1}=T(\min(A\setminus A_n))=2n$.
    From~\eqref{eq:lg_tij} we deduce that $\alpha_{n+1,n+1}=2\alpha_{1,n+1}-\alpha_{11}=4n-1$.
    Moreover, as $\alpha_{1,n+1}=2n$ and $\alpha_{2,n+1}=(\alpha_{22}+\alpha_{n+1,n+1})/2=2n+1$, we conclude
    that $\{1,2,\ldots,2n+1\}\subseteq T(A_{n+1})$. It remains to show that $2n+2\notin T(A_{n+1})$.
    Remark that the only even elements in $T(A_{n+1})$ are $\alpha_{1i}$ for $i=2,\ldots,n+1$.
    Indeed, the first condition in~\eqref{eq:lg_prop_ind} implies
    \[
        \alpha_{ij}=\frac{\alpha_{ii}+\alpha_{jj}}{2}=2(i+j)-5\qquad \text{for all $1<i\leq j\leq n+1$}.
    \]
    On the other hand, $\alpha_{1j}=(\alpha_{11}+\alpha_{jj})/2=2j-2$ for all $j=2,\ldots,n+1$. In particular,
    $2n+2$ is not contained in $T(A_{n+1})$.
\end{proof}

We fix the values $\alpha_{ij}$ from Proposition~\ref{prop:lg_tij}. 
The following statement is immediate.

\begin{lemma}
    \label{lm:lg_tij}
    If the bundle $\CU^\lambda$ is Ulrich, then $\alpha_{n+1-i}=\alpha_{ii}$ for $i=1,\ldots,n$.
\end{lemma}

\begin{proof}[Proof of Proposition~\ref{prop:c:max}]
    As $\alpha_1$ is the largest element in $\Irr(\lambda + \rho)$, one of the necessary conditions
    for $\CU^\lambda$ to be Ulrich is $\alpha_1=\dim X=n(n+1)/2$. On the other hand, Lemma~\ref{lm:lg_tij}
    together with Proposition~\ref{prop:lg_tij} impose $\alpha_1=\alpha_{nn}=4n-5$. Having solved the quadratic
    equation, one is left with two possible cases: $n=2$ or $n=5$. The former corresponds to
    $\alpha_1=3$ and $\alpha_2=1$, which in turn is equivalent to $\lambda=(1, 0)$. In the latter
    case $\alpha_3 = (\alpha_2+\alpha_4)/2$, thus $\Irr(\lambda)$ can not be of cardinality $d$.
\end{proof}

\subsubsection{Maximal orthogonal Grassmannians}
We are left with the case $X=\sG/\sP_n$, where $\sG$ is of type $B_n$ or $D_n$.
According to our convention, in both cases the lattice $P_\sG$ is identified with
$\ZZ^n\cup(\thalf+\ZZ)^n$, while the $\sL$-dominant weights precisely are those
$\lambda=a_1\omega_1+\ldots+a_n\omega_n$ for which $a_i\geq 0$ for $i=1,\ldots,n-1$
and $a_n$ is arbitrary. It is easy to see that $\lambda\in P^\sL$ if and only if
\[
    \lambda_1\geq\lambda_2\geq\cdots\geq\lambda_n.
\]
Now, the difference with the nonmaximal case comes from the fact that
\[
    \omega_n=\left(\thalf,\thalf,\ldots,\thalf\right).
\]
If one puts $\alpha=\lambda+\rho$, then
\[
    \Irr(\lambda)=\begin{cases}
        \{\alpha_i+\alpha_j\}_{1\leq i\leq j\leq n}, & \text{in type $B_n$,} \\
        \{\alpha_i+\alpha_j\}_{1\leq i<j\leq n}, & \text{in type $D_n$.}
    \end{cases}
\]

\begin{proposition}
    \label{prop:b:max}
    In type $B_n$ the only maximal isotropic Grassmannian carrying an irreducible
    Ulrich bundle is $\OGr(2,5)$. The corresponding bundle is trivial.
\end{proposition}
\begin{proof}
    Assume that $\CU^\lambda$ is Ulrich. In particular, $\Irr(\lambda)=\{1,\ldots,d\}$.
    We claim that $\CU^\mu$ is Ulrich on $\LGr(n,2n)$ for $\mu=2\lambda+2\rho_{B_n}-\rho_{C_n}$.
    Indeed, $\alpha'=\mu+\rho_{C_n}=2\alpha$ and $\dim\LGr(n,2n)=\dim\OGr(n,2n+1)$, thus
    \[
        \Irr_{C_n}(\mu)=\left\{\frac{\alpha'_i+\alpha'_j}{2}\right\}_{1\leq i\leq j\leq n} =
        \{\alpha_i+\alpha_j\}_{1\leq i\leq j\leq n} = \Irr_{B_n}(\lambda)=\{1,\ldots,d\}.
    \]
    It follows from Proposition~\ref{prop:c:max} that $n=2$ and $\mu=2\lambda+(1,0)=(1,0)$,
    from which $\lambda=0$.
\end{proof}

\begin{proposition}
    Every irreducible
    equivariant Ulrich bundle on $\OGr(2,4)$ and $\OGr(3,6)$ is trivial.
    The dual tautological bundle $\CU^*$ is the only irreducible equivariant Ulrich
    bundle on $\OGr(4,8)$.
    There are no irreducible equivariant vector bundles on~$\OGr(n,2n)$ for $n\geq 4$.
\end{proposition}
\begin{proof}
    We are basically interested in those $\alpha\in \ZZ^n\cup (\thalf+\ZZ)^n$ for which
    \(
        \alpha_1>\alpha_2>\cdots>\alpha_n,
    \)
    and
    \[
        \{\alpha_i+\alpha_j\}_{1\leq i<j\leq n}=\left\{1,\ldots,\tfrac{n(n-1)}{2}\right\}.
    \]
    An argument similar to the one used in type~$C_n$ shows that for such an~$\alpha$
    one necessarily has
    \[
        \alpha_n=0,\ \alpha_{n-1}=1,\ \alpha_{n-2}=2,\ \alpha_{n-3}=4,\ \alpha_{n-4}=7,\ \alpha_{n-5}=10.
    \]
    Then $\alpha_{n-3}+\alpha_{n-4}=\alpha_{n-1}+\alpha_{n-5}$, so $n\leq 5$. On the other
    hand, for $n=5$ the set $\Irr=\{1,\ldots,9,11\}$ does not contain $d$. It remains to check
    (say, by hand) that $\lambda=0$ for $n=2,3$ and $\lambda=(1,0,0,0)$ for $n=4$ indeed provide
    Ulrich bundles.
\end{proof}

\begin{remark}
    Of course, the negative results obtained in this paper do not contradict the original
    conjecture. In~\cite{ulrichFl} the authors show that most of the flag varieties in type~$A$
    do not carry an irreducible Ulrich bundle. On the other hand, according to our experience,
    it is very limiting to restrict oneself to the class of irreducible bundles. There seem
    to be equivariant bundles that are not irreducible, but still enjoy incredibly nice
    properties (see, e.g.,~\cite{fonarevKP}). It is also interesting to see if our results
    can help construct Ulrich bundles on other varieties, e.g. those K\"uchle varieties that
    come from sections of Lagrangian Grassmannians of planes~(see~\cite{kuznetsovKuh}).
\end{remark}

\printbibliography

\end{document}